\documentclass[final]{article}

\usepackage[margin=1in]{geometry}
\usepackage{amsmath}
\usepackage{amssymb}
\usepackage{amsthm}
\usepackage{graphicx}
\usepackage{epstopdf}

\usepackage{tabularx}
\usepackage{subfigure,booktabs,multirow,listings,paralist}

\newcommand\R{\mathcal R}
\newcommand\N{\mathcal N}

\newtheorem{theorem}{Theorem}

\newtheorem{definition}[theorem]{Definition}
\newtheorem{observe}[theorem]{Observation}
\newtheorem{remark1}[theorem]{Remark}

\newenvironment{remark}{\begin{remark1} \rm}{\end{remark1}}

\usepackage{latexsym} 

    \newcolumntype{Y}{>{\small\center\arraybackslash}X}
    \newcolumntype{W}{>{\small$}c<{$}} 
    \newcolumntype{Z}{>{\small\centering\arraybackslash$}X<{$}}

\title{Randomized methods for rank-deficient linear systems}
\author{Josef Sifuentes\thanks{%
Department of Mathematics,
Texas A\&M University,
Mailstop 3368,
College Station, TX 77843-3368.
{{\em email}: {\sf josefs@math.tamu.edu.}}}
, Zydrunas Gimbutas\thanks{%
Information Technology Laboratory,
National Institute of Standards and Technology,
325 Broadway, Mail Stop 891.01,
Boulder, CO  80305-3328.
{{\em email}: {\sf {zydrunas.gimbutas@nist.gov}}}. The work of this author was 
supported in part by the 
Office of the Assistant Secretary of Defense for Research and Engineering 
and AFOSR under NSSEFF Program Award FA9550-10-1-0180.
and in part by the National Science Foundation under
grant DMS-0934733. Contributions by staff of NIST, an agency of the U.S. Government, 
are not subject to copyright within the United States.}
, and Leslie Greengard\thanks{Courant Institute of Mathematical Sciences,
         New York University, 
         251 Mercer Street,
         New York, NY 10012-1110.
{{\em email}: {\sf {greengard@cims.nyu.edu}}}. The work of this author was 
supported in part by the 
Office of the Assistant Secretary of Defense for Research and Engineering 
and AFOSR under NSSEFF Program Award FA9550-10-1-0180,
in part by the National Science Foundation under
grant DMS-0934733, and in part
by the Applied Mathematical Sciences Program of the U.S. Department of Energy
under Contract DEFGO288ER25053.}}

\begin{document}

\maketitle

{\footnotesize {\bf Abstract.} 
We present a simple, accurate method for solving consistent, rank-deficient 
linear systems, with or without additional rank-completing constraints. 
Such problems arise in a variety 
of applications, such as the computation of the eigenvectors of a matrix corresponding 
to a known eigenvalue. The method is based on elementary linear algebra combined with the 
observation that if the matrix is rank-$k$ deficient, then a random rank-$k$
perturbation yields a nonsingular matrix with probability 1.}\\


\section{Introduction}

A variety of problems in numerical linear algebra involve the 
solution of rank-deficient linear systems.
The most straightforward example is that of finding the eigenspace of a
matrix $A \in \mathbb{C}^{n \times n}$ corresponding to a known eigenvalue
$\lambda$. One then wishes to solve:
\[ (A - \lambda I) x = 0. \]
If $A$ itself is rank-deficient, of course,
then setting $\lambda = 0$ corresponds to seeking its nullspace.

A second category of problems involves the solution of an 
inhomogeneous linear system
\begin{equation}
\label{axb}
 A x = b, 
\end{equation}
where $A$ is rank-$k$ deficient but $b$ is in the range of $A$.
A third category consists of problems like $(\ref{axb})$,
but for which a set of $k$ additional
constraints are known of the form:
\begin{equation}
\label{cxf}
 C^* x = f \, ,
\end{equation}
where
\[   \left(\begin{array}{c} A \\C^* \end{array}\right) \]
is full-rank.
Here, $C \in \mathbb{C}^{n \times k}$, $C^*$ denotes its adjoint,
and $f \in \mathbb{C}^{k}$.

In this brief note, we describe a very simple framework for solving such
problems, using {\em randomized} schemes. 
They are particularly useful
when $A$ is well-conditioned in a suitable $(n-k)$-dimensional subspace.
In terms of the singular value decomposition
$A = U \Sigma V^*$, 
this corresponds to the case when 
$\sigma_1/\sigma_{n-k}$ is of modest
size and $\sigma_{n-k+1},\dots,\sigma_n = 0$, where 
the $\{ \sigma_i \}$ are the singular values of $A$.
We do {\em not} address
least squares problems, that is, we assume that the system
(\ref{axb}), with or without(\ref{cxf}), is consistent. 

\begin{definition}
We will denote by $\N(A)$ the nullspace of $A$ and
by $\R(A)$ its range.
\end{definition}

There is a substantial literature on this subject, which we do not seek
to review here. We
refer the reader to the texts \cite{GVL,HANSEN2} and the papers
\cite{BV,CHAN,CP,CP2,DHILLON,GOLUB,HANSEN1,Hochstenbach,IPSEN,STEWART}.
Of particular note are 
\cite{PAN1,PAN2}, which demonstrate the power of randomized schemes
using methods closely related to the ones described 
below. It is also worth noting that, in recent years, the use of
randomization together with numerical rank-based ideas 
has proven to be a powerful combination for a variety of problems
in linear algebra (see, for example, \cite{HMT,Tygert1,Tygert2}).

The basic idea in the present work is remarkably simple and illustrated by 
the following example. Suppose we are given a rank-1 deficient 
matrix $A$ and that we carry out the following procedure:

\begin{enumerate}
\item Choose a random vector $x \in \mathbb{C}^n$ and compute $b = Ax$.
\item Choose random vectors $p,q \in \mathbb{C}^n$ and solve
\begin{eqnarray}
(A + p q^*) y = b.
\label{eqn.solve}
\end{eqnarray}
\item Then the difference $x - y$ is in the nullspace of $A$.
\end{enumerate}

In order for $A + p q^*$ to be invertible, we must have that
$p \notin \mathcal{R}(A)$ and 
$q \notin \mathcal{R}(A^*)$. Since $p$ and $q$ are random, this 
must occur with probability 1.
It follows then that
$A(x-y) = b - (b - pq^{*}y) = p(q^{*}y)$. Since 
$A(x-y)$ must be in $\mathcal{R}(A)$ and $p$ is not,
both sides vanish, implying that $x-y$ is a null-vector of A, and 
$q^*y$ must be zero.

Another perspective, which may be more natural to some readers,
is to consider the affine space
$\{x' + \mathcal{N}(A) \}$, consisting of solutions to $A z = b$, where, $x'$
 is the solution of minimal norm. The difference of any two vectors 
in the affine space clearly lies in the nullspace of $A$.  If $A + p q^*$ is 
 nonsingular, $y$ is the unique vector in the affine space 
 orthogonal to $q$, implying that $x - y \in \mathcal{N}(A)$.


The remainder of this note is intended to make this 
procedure rigorous and to explore its extensions to related problems
such as solving (\ref{axb}), (\ref{cxf}).

\section{Mathematical preliminaries}

Much of our analysis depends on estimating the condition number
of a rank-$k$ deficient complex $n \times n$ matrix $A$ to which is added a 
rank-$k$ random perturbation. 
For $P,Q \in \mathbb{C}^{n\times k}$, we let

\begin{eqnarray}
P &=& P_R + P_{N^*}, \qquad \mathcal{R}(P_R) \subset \mathcal{R}(A),  \, \mathcal{R}(P_{N^*}) \subset \mathcal{N}(A^*), \nonumber \\
Q &=& Q_{R^*} + Q_{N}, \qquad \mathcal{R}(Q_{R^*}) \subset \mathcal{R}(A^*),  \, \mathcal{R}(Q_N) \subset \mathcal{N}(A),
\label{PNQNdef}
\end{eqnarray}
and
\begin{eqnarray}
\rho := \|  P_R \| = \sigma_{max}(P_R), &\qquad& \eta := \sigma_{min}( P_{N^*} ), \nonumber \\
\xi := \| Q_{R^*} \| = \sigma_{max}(Q_{R^*}), &\qquad& \nu := \sigma_{min} ( Q_N ), \label{rhoetadef}
\end{eqnarray}
where all norms $\| \cdot \| = \| \cdot \|_2$. 



\begin{theorem}
Let $b = Ax$ and let ${y}$ be an approximate 
solution to 
\begin{equation}
(A + PQ^*)  y = b
\end{equation}
in that it satisfies
\begin{eqnarray}
\| b - (A + PQ^*) {y} \| \le \delta.
\label{eqn.solveres}
\end{eqnarray}  
Then
\begin{eqnarray}
\| A (x- {y}) \| \le \delta \left( 1 + \frac{\| P \|}{\sigma_{min}( P_{N^*} )} \right).
\label{eqn.errorbound}
\end{eqnarray}
\end{theorem}

\begin{proof}
It follows from (\ref{eqn.solveres}) and the triangle inequality that
\begin{eqnarray}
\| A(x - {y}) \| \le \delta + \| P \| \|Q^* {y} \|.
\label{eqn.resstep}
\end{eqnarray}
Moreover,
\begin{eqnarray}
b - A {y} - P (Q^* {y}) = \delta f
\end{eqnarray}
for some vector $f \in \mathbb{C}^n$ with $\| f \| \le 1$. 
Now let $U$ be a matrix whose columns form an orthonormal basis for $\N(A^*)$. 
Multiplying on the left by $U^*$ we have 
\begin{eqnarray}
-(U^*P) \, (Q^* {y}) &=& \delta (U^* f), \\
\|Q^*y\| &\le& \frac{\delta}{ \sigma_{min}(P_{N^*}) },
\end{eqnarray}
where the last inequality follows from the fact that 
\begin{equation}
\delta \ge \inf_{\|z\|=1,z \in \mathbb{C}^k} \| U^* P z \| \| Q^* y \| = 
\inf_{\|z\|=1,z \in \mathbb{C}^k} \| U U^* P z \| \| Q^* y \|
= \sigma_{min} (P_{N^*}) \| Q^* y \|,
\end{equation} 
which yields the desired result when combined with (\ref{eqn.resstep}).
\end{proof}

The obtained bound (\ref{eqn.errorbound}) indicates that $x-y$ is 
an approximate null-vector of matrix $A$, therefore, $y$ is also 
an approximate solution to $A y = b$ for a given
consistent right-hand side $b\in \R(A)$.

\begin{theorem}
Let $A \in \mathbb{C}^{n \times n}$ have a $k$-dimensional nullspace and 
let $P,\,Q \in \mathbb{C}^{n\times k}$.  Then
\begin{eqnarray}
\| (A + PQ^*)^{-1} \| &\le& \frac{1}{\sigma_{n-k}} 
\sqrt{
1 + 
 \left(\frac{\rho}{\eta}\right)^2 + 
 \left(\frac{\xi}{\nu}\right)^2 + 
 \left( \frac{\sigma_{n-k} + \rho \xi}{\eta \nu} \right)^2  \, ,
 }
\end{eqnarray}
\label{the.geninvbound.multi}
where $\rho,\eta,\xi,\nu$ are defined in (\ref{rhoetadef}).
\end{theorem}

\begin{proof}
\vspace{0.1in}

Let $A = U \Sigma V^*$ be the singular value decomposition of $A$.
Let $C$ and $D$ be such that $P = UC$ and $Q = VD$.  
Let $C^T = [C_R^T \,\,\, C_{N^*}^T]$ where $C_R \in \mathbb{C}^{(n-k) \times k}$
and $C_{N^*} \in \mathbb{C}^{k \times k}$.  The entries in the columns of $C_R$ 
are coefficients of the corresponding columns of $P$ in an orthonormal basis of the range of $A$.  
Thus $\| C_R \| = \rho$, and similarly $\| C_{N^*}^{-1} \| = 1 / \eta$.  Let 
$D^T = [D_{R^*}^T \,\,\, D_N^T]$ where $D_{R^*} \in \mathbb{C}^{(n-k) \times k}$
and $D_N \in \mathbb{C}^{k \times k}$.  By similar reasoning, we have that 
$\| D_{R^*} \| = \xi$ and $\| D_N^{-1} \| = 1 / \nu$
\begin{eqnarray}
\| (A + P Q^*)^{-1} \| = \| ( \Sigma + C D^* )^{-1} \|
\end{eqnarray}
and
\begin{eqnarray}
(\Sigma + CD^*)^{-1} 
&=& 
\left(
\begin{array}{cc}
\Sigma' + C_R D_{R^*}^* & C_R D^*_N \\
C_{N^*} D_{R^*}^* & C_{N^*} D^*_N
\end{array}
\right)^{-1} \\
&=&
\left(
\begin{array}{cc}
\Sigma'^{-1} & -\Sigma'^{-1} C_R (C_{N^*})^{-1}\\
-(D^*_N)^{-1} D^*_{R^*}\, \Sigma'^{-1} & 
(D_N^*)^{-1} \left(I_k + D_{R^*}^* \, \Sigma'^{-1} C_R\right) (C_{N^*})^{-1}
\end{array}
\right),
\end{eqnarray}
where $\Sigma' \in \mathbb{C}^{(n-k) \times (n-k)}$ is the upper $(n-k) \times (n-k)$ sub-matrix of $\Sigma$ and $I_k \in
\mathbb{C}^{k \times k}$ is the identity matrix.  This gives
\begin{eqnarray}
\| ( \Sigma + C D^*)^{-1} \| &\le&
\sqrt{
\frac{1}{\sigma_{n-k}^2} + 
\left(\frac{\rho}{\sigma_{n-k} \eta}\right)^2 + 
\left( \frac{\xi}{\sigma_{n-k} \nu} \right)^2 +
\left(
\frac{1 + \rho \xi /  \sigma_{n-k}}{\eta \nu} 
\right)^2
} \\
&=& \frac{1}{\sigma_{n-k}} \sqrt{
1 + \left(\frac{\rho}{\eta}\right)^2 + 
\left(\frac{\xi}{\nu}\right)^2 + 
\left(\frac{\sigma_{n-k} + \rho \xi}{\eta \nu}\right)^2.
}
\end{eqnarray}

\end{proof}

It follows from this result that one can bound the conditioning of the 
perturbed matrix.
\begin{theorem} \label{condthm}
Let $A \in \mathbb{C}^{n \times n}$ have a $k$-dimensional nullspace and 
let $P, \, Q  \in \mathbb{C}^{n \times k}$.  Then
\begin{eqnarray}
\kappa(A + P Q^*) &\le& 
\frac{\sigma_1 + \| P \| \, \| Q \|}{\sigma_{n-k}} \sqrt{
1 + \left(\frac{\rho}{\eta}\right)^2 + 
\left(\frac{\xi}{\nu}\right)^2 + 
\left(\frac{\sigma_{n-k} + \rho \xi}{\eta \nu}\right)^2 \, ,
}
\end{eqnarray}
where $\rho,\eta,\xi,\nu$ are defined in (\ref{rhoetadef}).
\end{theorem}

The preceding theorems indicate that, in the absence of
additional information, it is reasonable to pick random vectors of
approximately unit norm and to scale the perturbation term $P Q^{*}$
by the norm of $A$.

\section{Solving consistent, rank-deficient linear systems} \label{sec:consist}

Let us first consider the solution of the rank-$k$ deficient linear system
$Ax=b$ in the special case where $\N(A)$ and $\N(A^*)$ are spanned by the  
columns of known matrices $N$ and $V$, respectively. 
Suppose, now, that we solve the linear system
\begin{equation}
 (A + VN^*)x = b \, . 
\label{rankdefsys}
\end{equation}
Consistency here requires that $V^*Ax = V^*b = 0$, so that 
$(V^*V)(N^*x) = 0$, from which $N^*x = 0$. Thus, $x$ is the particular
solution to $Ax=b$ that is orthogonal to the nullspace of $A$.
From Theorem \ref{condthm}, the condition number of $A+VN^*$ is given by 
\begin{eqnarray}
\kappa(A + V N^*) &\le& 
\frac{\sigma_1 + \| V \| \, \| N \|}{\sigma_{n-k}} \sqrt{
1 + 
\left(\frac{\sigma_{n-k}}{\sigma_{min}(V) \sigma_{min}(N)}\right)^2
}.
\end{eqnarray}

\begin{remark}
Note that this procedure allows us to 
obtain the minimum norm solution to the underdetermined
linear system without recourse to the SVD or other dense matrix methods.
Any method for solving (\ref{rankdefsys}) can be used. Assuming that 
$(A+VN^*)$ is reasonably well conditioned and that $A$ can be applied 
efficiently, Krylov space methods such as GMRES are extremely effective.
\end{remark}

Suppose now that we have no prior information about the nullspaces
of $A$ and/or $A^*$. We may then substitute random matrices
$P$ and $Q$ for 
$V$ and/or $N$ and follow the same procedure. With probability 1, 
$(A + P Q^*)$ will be invertible and we will obtain the particular
solution to $Ax=b$ that is orthogonal to $Q$. This simply requires
that the projections of $P$ onto $\N(A^*)$
and of $Q$ onto $\N(A)$, denoted by 
$P_{N^*}$ and $Q_{N}$ respectively, must be full-rank
(see (\ref{PNQNdef})). 

\subsection{Consistent, rectangular linear systems} \label{sec:rect}

We next consider the case where we wish to solve
the system (\ref{axb}) together with (\ref{cxf}). Note that, for consistency, 
we must still have that
$V^*A = V^*b = 0$, where the columns of $V$ span $\N(A^*)$.
Note also that the system 
\begin{equation}
   \left(\begin{array}{c} A \\C^* \end{array}\right) x = 
\left(\begin{array}{c} b \\f \end{array} \right)  \label{rectangular}
\end{equation}
is full-rank if and only if 
any vector in $\N(A)$ has a nontrivial projection onto the 
columns of $C$.
There is no need, however, to solve a rectangular system
of equations (\ref{rectangular}).  One need only solve the $n \times n$ linear system
\[  (A + V C^*) x = b + V f\, .\]
From Theorem \ref{condthm}, the condition number of $A+VC^*$ is given by 
\begin{eqnarray}
\kappa(A + V C^*) &\le& 
\frac{\sigma_1 + \| V \| \, \| C \|}{\sigma_{n-k}} \sqrt{
1 + \left(\frac{\xi}{\sigma_{min}(C_{N})}\right)^2 + 
\left(\frac{\sigma_{n-k}}{\sigma_{min}(V) \sigma_{min}(C_{N})}\right)^2
},
\end{eqnarray}
where $\xi$ is the norm of $C_{R^*}$.
In some applications, the data may be known to be consistent
($b$ is in the range of $A$), but $V$ may not be known.
Then, one can proceed, as above, by solving 
\[  (A + P C^*) x = b + P f\, ,\]
where $P$ is a random $n \times k$ matrix.
From Theorem \ref{condthm}, the condition number of $A+PC^*$ is given by 
\begin{eqnarray}
\kappa(A + P C^*) &\le& 
\frac{\sigma_1 + \| P \| \, \| C \|}{\sigma_{n-k}} \sqrt{
1 + \left(\frac{\rho}{\sigma_{min}(P_{N^*})}\right)^2 + 
\left(\frac{\xi}{\sigma_{min}(C_{N})}\right)^2 + 
\left(\frac{\sigma_{n-k} + \rho \xi}{\sigma_{min}(P_{N^*}) \sigma_{min}(C_{N})}\right)^2
},
\end{eqnarray}
where $\rho$ and $\xi$ are the norms of $P_{R}$ and 
$C_{R^*}$, respectively.

\section{Computing the nullspace}

Let us return now to the question of finding a basis for the 
nullspace of a rank-$k$ deficient matrix
$A \in \mathbb{C}^{n \times n}$. As in the introduction, we begin 
by describing the procedure.

\begin{enumerate}
\item Choose $k$ random vectors $\{x_i, i = 1,\dots k\}
\in \mathbb{C}^n$ and compute $b_i = Ax_i$.
\item Choose random matrices $P,Q \in \mathbb{C}^{n\times k}$ and solve
\begin{eqnarray}
(A + P Q^*) y_i = b_i.
\label{eqn.solvesys}
\end{eqnarray}
\end{enumerate}

Then, $A(x_i-y_i) = b_i - (b_i - PQ^* y_i) = P(Q^*y_i)$. Since
$A(x_i-y_i) \in \R(A)$, and assuming $P(Q^*y_i) \notin \R(A)$, it follows
that both sides must equal zero and that
each vector $z_i = x_i-y_i$ is a null vector. 
Since the construction is random, the probability that the $\{z_i\}$ are
linearly independent is 1. 
The result $P(Q^*y_i) \notin \R(A)$ follows from the fact that
$P$ is random and that the projection of each column of $P$ onto 
$\N(A^*)$ will be linearly independent with probability 1.
Theorem \ref{condthm} tells us how to estimate the condition number
of (\ref{eqn.solvesys}). Finally, the accuracy of the nullspace vectors $\{z_i\}$ can be further
improved by an iterative refinement $\tilde z_i = z_i-\tilde y_i$, where the correction vectors ${\tilde y_i}$ solve (\ref{eqn.solvesys})
\begin{eqnarray}
  (A + P Q^*) \tilde y_i = \tilde b_i, 
\label{eqn.solvesys_iter}
\end{eqnarray}
with the updated right-hand sides $\tilde b_i =  Az_i$.

\subsection{Determining the dimension of the nullspace}

When the dimension of the nullspace is unknown,
the algorithm above can also be used as a {\em rank-revealing} scheme.
For this, suppose that the actual rank-deficiency is known to be
$k_A$ and that we carry out the above procedure with $k > k_A$. 
The argument that 
$P(Q^*y_i) \notin \R(A)$ will fail, since the 
projection of each of the columns of $P$ onto 
$\N(A^*)$ must be linearly dependent. As a result, $x_i-y_i$ will fail to be 
a null-vector (which will be obvious from the explicit computation of 
$A(x_i-y_i)$. The estimated rank $k$ can then be systematically reduced
to determine $k_A$. If $k_A$ is large, bisection can be used to accelerate
this estimate. 

\subsection{Stabilization}

Since the condition number of the randomly perturbed matrix is 
controlled only in a probabilistic sense, if high precision is required
one can use a variant of iterative refinement to improve the solution.
That is, one can first compute $q_1,\dots,q_k$ as approximate
null-vectors of $A$ and 
$p_1,\dots,p_k$ as approximate
null-vectors of $A^*$. With these at hand, one can repeat the calculation
with $P$ and $Q$ whose columns are $\{p_1,\dots,p_k \}$
and $\{q_1,\dots,q_k \}$, respectively. The parameters $\rho/\eta$ and 
$\xi/\nu$ in Theorem \ref{condthm} will be much less than 1
and the condition number of a second iteration will be approximately

\begin{eqnarray}
\kappa(A + P Q^*) &\approx& 
\frac{\sigma_1 + \| P \| \, \| Q \|}{\sigma_{n-k}} \sqrt{
1 + \left(\frac{\sigma_{n-k}}{\sigma_{min}(P_{N^*}) \sigma_{min}(Q_{N})}\right)^2
} \, .
\end{eqnarray}

\section{Numerical experiments}

In this section, we describe the results of several numerical tests of
the algorithms discussed above. All computations were performed in
IEEE double-precision arithmetic using MATLAB version R2012a \footnote{Any mention of
commercial products or reference to commercial organizations is for
information only; it does not imply recommendation or endorsement by
NIST.}.

We use a pseudorandom number generator to create
$n\times 1$ vectors $\phi_1$, $\phi_2, \ldots, \phi_{n-k}$ and
$\psi_1$, $\psi_2, \ldots, \psi_{n-k}$, with entries that are
independent and identically distributed Gaussian random variables of
zero mean and unit variance. We apply
the Gram-Schmidt process with reorthogonalization to $\phi_1$,
$\phi_2, \ldots, \phi_{n-k}$ and $\psi_1$, $\psi_2, \ldots,
\psi_{n-k}$ to obtain orthonormal vectors $u_1, u_2, \ldots, u_{n-k}$,
and $v_1, v_2, \ldots, v_{n-k}$, respectively. We define $A$ to be the
$n\times n$ matrix
\begin{equation}
   A = \sum_{i=1}^{n-k} u_i \sigma_i v_i^*,
\end{equation}
where $\sigma_i = 1/i$. The rank deficiency of A is clearly equal to $k$.

In Table \ref{tab:ex1simple}, we compare the regular and stabilized
versions of the new algorithm for finding the nullspace of a
rank-deficient matrix $A$. The first and second columns contain the
parameters $n$ and $k$ determining the size and the rank deficiency of
problem, respectively.  The third column contains the modified
condition number $\sigma_1/\sigma_{n-k}$ of the original matrix $A$,
ignoring the zero singular values for more meaningful comparison
between columns. The fourth columns contains the true condition number
$\sigma_1/\sigma_{n}$ of a random rank-$k$ perturbation $A+P
Q^{*}$. Finally, the fifth and sixth columns contain the relative
accuracy $||AN||/||N||$ in determining the nullspace vectors $N$ for
the randomized rank-$k$ correction scheme before and after iterative
refinement, respectively.

In Table \ref{tab:ex2}, we compare the accuracy of the regular and
stabilized versions of the randomized rank-k correction scheme for
solving a rank-deficient linear system $Ax=b$ with a consistent right
hand side $b$. The first and second columns contain the parameters $n$
and $k$ determining the size and the rank deficiency of problem,
respectively.  The third and forth columns contain the modified
condition number $\sigma_1/\sigma_{n-k}$ of the original matrix $A$
and the condition number $\sigma_1/\sigma_{n}$ of a random rank-$k$
perturbation $A+P Q^{*}$, respectively. The fifth columns contains the
condition number $\sigma_1/\sigma_{n}$ of the rank-$k$ perturbation
$A+V N^{*}$, where $V$ and $N$ are the approximate null-vectors
spanning the left and right nullspaces, respectively.  Finally, the
fifth and seventh columns contain the relative accuracy
$||Ax-b||/||b||$ in determining the solution vector $x$ for the
regular and stabilized schemes, respectively.

It is clear from Table \ref{tab:ex2} that the condition number can be
quite large for the non-stabilized version of the algorithm when the
rank deficiency is high. This is due to the difficulty of finding
high-dimensional random matrices $P$ and $Q$ that have large
projections onto the corresponding nullspaces $\N(A^{*})$ and $\N(A)$.
In such cases, the algorithm will strongly benefit from
the stabilization procedure.

\begin{table}
     \centering
     \begin{tabularx}{.9\linewidth}{Z|Z Z Z Z Z Z} \toprule
  n  & k & \text{cond}(A) & \text{cond}(A+P Q^{*})  
                               & E_2 & E_2(refined)  \\
\midrule

 160 &  1 & 1.590E+02 & 2.025E+03 & 1.368E-16 & 8.106E-17 \\ 
 160 &  3 & 1.570E+02 & 4.258E+04 & 2.180E-15 & 2.727E-16 \\ 
 160 &  6 & 1.540E+02 & 1.144E+04 & 2.706E-14 & 6.382E-16 \\ 
 320 &  1 & 3.190E+02 & 5.259E+03 & 9.072E-17 & 3.556E-17 \\ 
 320 &  3 & 3.170E+02 & 9.340E+03 & 1.983E-16 & 6.029E-17 \\ 
 320 &  6 & 3.140E+02 & 3.374E+04 & 7.461E-16 & 2.471E-16 \\ 
 640 &  1 & 6.390E+02 & 3.968E+04 & 1.934E-16 & 2.099E-16 \\ 
 640 &  3 & 6.370E+02 & 1.332E+06 & 3.879E-15 & 5.817E-16 \\ 
 640 &  6 & 6.340E+02 & 3.899E+06 & 5.924E-13 & 5.781E-16 \\ 
1280 &  1 & 1.279E+03 & 6.003E+06 & 5.549E-16 & 3.244E-16 \\ 
1280 &  3 & 1.277E+03 & 4.998E+04 & 1.023E-14 & 6.990E-17 \\ 
1280 &  6 & 1.274E+03 & 6.515E+05 & 3.706E-15 & 8.126E-16 \\ 
\midrule
 160 & 75 & 8.500E+01 & 2.394E+05 & 4.208E-13 & 2.118E-14 \\ 
 160 & 80 & 8.000E+01 & 3.199E+04 & 2.185E-13 & 2.480E-15 \\ 
 320 & 155 & 1.650E+02 & 1.445E+06 & 3.228E-12 & 7.493E-15 \\ 
 320 & 160 & 1.600E+02 & 1.607E+06 & 1.465E-11 & 1.578E-14 \\ 
 640 & 315 & 3.250E+02 & 1.006E+07 & 1.155E-11 & 6.883E-15 \\ 
 640 & 320 & 3.200E+02 & 4.288E+06 & 1.602E-11 & 1.939E-14 \\ 
1280 & 635 & 6.450E+02 & 3.551E+08 & 2.714E-10 & 4.323E-14 \\ 
1280 & 640 & 6.400E+02 & 1.873E+08 & 1.902E-11 & 5.665E-14 \\ 
        \bottomrule
          \end{tabularx}
  \caption{\em\small Relative errors in determining the nullspace vectors 
for the randomized rank-k correction scheme before and after iterative
refinement.}
\label{tab:ex1simple}
\end{table}

\begin{table}
     \centering
     \begin{tabularx}{.9\linewidth}{Z|Z Z Z Z Z Z} \toprule
  n  & k & \text{cond}(A) & \text{cond}(A+P Q^{*})  
                               & E_2& \text{cond}(A+U V^{*})  
                               & E_2(stab) \\
\midrule
 160 &  1 & 1.590E+02 & 9.017E+02 & 1.282E-15 & 1.590E+02 & 1.141E-15 \\ 
 160 &  3 & 1.570E+02 & 3.121E+03 & 3.890E-15 & 1.570E+02 & 1.910E-15 \\ 
 160 &  6 & 1.540E+02 & 1.284E+06 & 1.487E-13 & 1.540E+02 & 1.656E-15 \\ 
 320 &  1 & 3.190E+02 & 4.956E+05 & 7.388E-15 & 3.190E+02 & 1.209E-15 \\ 
 320 &  3 & 3.170E+02 & 4.059E+05 & 6.638E-14 & 3.170E+02 & 2.939E-15 \\ 
 320 &  6 & 3.140E+02 & 3.271E+04 & 1.100E-14 & 3.140E+02 & 2.704E-15 \\ 
 640 &  1 & 6.390E+02 & 1.232E+05 & 1.758E-14 & 6.390E+02 & 2.072E-15 \\ 
 640 &  3 & 6.370E+02 & 8.812E+04 & 9.113E-15 & 6.370E+02 & 3.085E-15 \\ 
 640 &  6 & 6.340E+02 & 1.622E+05 & 9.870E-15 & 6.340E+02 & 2.797E-15 \\ 
1280 &  1 & 1.279E+03 & 8.325E+04 & 4.545E-15 & 1.279E+03 & 3.483E-15 \\ 
1280 &  3 & 1.277E+03 & 5.174E+05 & 1.714E-14 & 1.277E+03 & 6.914E-15 \\ 
1280 &  6 & 1.274E+03 & 7.675E+05 & 3.905E-14 & 1.274E+03 & 4.661E-15 \\ 
\midrule
 160 & 75 & 8.500E+01 & 7.057E+04 & 3.854E-13 & 8.500E+01 & 4.157E-15 \\ 
 160 & 80 & 8.000E+01 & 2.357E+04 & 9.249E-14 & 8.000E+01 & 3.975E-15 \\ 
 320 & 155 & 1.650E+02 & 1.732E+05 & 1.886E-13 & 1.650E+02 & 1.208E-14 \\ 
 320 & 160 & 1.600E+02 & 9.449E+05 & 6.109E-12 & 1.600E+02 & 8.945E-15 \\ 
 640 & 315 & 3.250E+02 & 5.510E+07 & 8.537E-11 & 3.250E+02 & 2.612E-14 \\ 
 640 & 320 & 3.200E+02 & 2.623E+07 & 1.591E-11 & 3.200E+02 & 1.884E-14 \\ 
1280 & 635 & 6.450E+02 & 5.970E+06 & 7.540E-12 & 6.450E+02 & 3.236E-14 \\ 
1280 & 640 & 6.400E+02 & 1.134E+07 & 1.162E-11 & 6.400E+02 & 7.476E-14 \\ 
        \bottomrule
          \end{tabularx}
          \caption{\em\small Relative errors for the regular and stabilized versions of the randomized rank-k correction scheme in determining the solution of the rank-$k$ deficient linear system $Ax=b$ with the consistent right-hand side $b\in \R(A)$.}
\label{tab:ex2}
\end{table}

\section{Further examples}

Our interest in the development of randomized methods was driven largely by 
issues in the regularization of integral equation methods in potential theory.
For illustration, consider the Neumann problem for the Laplace equation
in the interior of a simply-connected, smooth domain $\Omega \subset \mathbb{R}^2$ with
boundary $\Gamma$.

\[ \Delta u = 0 \ {\rm in}\ \Omega,\quad
\frac{\partial u}{\partial n} = f \ {\rm on}\ \Gamma \, .
\]

Classical potential theory \cite{GL} suggests seeking the solution as a single layer
potential
\[ u(x) = \frac{1}{2\pi} \int_\Gamma \log\|x-y\| \sigma(y) \, ds_y \, . \]
Using standard jump relations, this results in the integral equation
\begin{equation} \sigma(x) + \frac{1}{\pi} \int_\Gamma \frac{\partial}{\partial n_x}
\log\|x-y\| \sigma(y) \, ds_y = 2f(x) \, , 
\label{neuinteq}
\end{equation}
which we write as
\[ (I + K) \sigma = 2f \, . \]
It is well-known that (\ref{neuinteq}) is solvable if and only if the right-hand
side satisfies the compatibility condition: $\int_\Gamma f(y) ds_y = 0$.
Using the $L_2$ inner product (for real-valued functions)
\[ \langle f,g \rangle =  \int_\Gamma f(y) g(y) ds_y, \]
we may write the compatibility condition as 
\[ {\langle 1,f \rangle} =  0 \, , \]
where $1$ denotes the function that is identically $1$ on $\Gamma$.
The function $1$ is also in the nullspace of $I + K^*$, the adjoint of the integral
operator in (\ref{neuinteq}), which is clearly neccesary for solvability.
Following the procedure in section \ref{sec:consist}, we
regularize the integral equation by solving

\begin{equation} \sigma(x) + \frac{1}{\pi} \int_\Gamma \frac{\partial}{\partial n_x}
\log\|x-y\| \sigma(y) \, ds_y  +  \int_\Gamma [r(x) {1}(y) ] \sigma(y) \, dy
= 2f(x) \, , 
\label{neuinteqreg}
\end{equation}
or 
\[ (I + K)\sigma + r(x) {\langle 1, \sigma \rangle}  = 2f \, , \]
where $r(x)$ is a random function defined on $\Gamma$.
Taking the inner product of (\ref{neuinteqreg}) with the function $1$ yields
\[ \langle 1, r \rangle \, {\langle 1, \sigma \rangle}  = 0 \, . \]
This is a well-known fact for the Neumann problem, and the obvious choice
is simply $r(x) = 1$ so that (\ref{neuinteqreg}) becomes:
\[ \sigma(x) + \frac{1}{\pi} \int_\Gamma  \left[ \frac{\partial}{\partial n_x}
\log\|x-y\| + 1 \right] \sigma(y) \, ds_y = 2f(x) \, .
\]

For an application of the preceding analysis
in electromagnetic scattering, see \cite{mfie_stab}.
In \cite{mfie_mc}, a situation of the type 
discussed in section \ref{sec:rect} arises. 
Without entering into details, it was shown that the
``magnetic field integral equation" 
is rank-$k$ deficient in the static limit
in exterior multiply-connected domains of genus $k$. A set of $k$
nontrivial constraints was derived from electromagnetic considerations,
which were added to the system matrix as described above. 
Since we have illustrated the basic principle in the context of the nullspace
problem, we omit further numerical calculations.

\section{Conclusions}

We have presented a simple set of tools for solving rank-deficient,
but consistent, linear systems and demonstrated their utility with
some numerical examples. Since the perturbed/augmented linear systems
are reasonably well-conditioned with high probability, one can rely on
Krylov subspace based iterative methods (e.g., conjugate gradient for
self-adjoint problems or GMRES for non self-adjoint problems),
avoiding the cost of dense linear algebraic methods, such as Gaussian
elimination or the SVD itself. This is a particularly powerful
approach when $A$ is sparse or there is a fast algorithm for applying
$A$ to a vector.  Finite rank-deficiency issues arise in the
continuous setting as well, especially in integral equation methods,
which we have touched on only briefly here.

We are currently working on the development of robust software for the
nullspace problem that we expect will be competitive 
with standard approaches such as QR-based schemes \cite{CHAN},
inverse iteration \cite{DHILLON,GVL} or Arnoldi methods \cite{GOLUB}.

\section{Acknowledgment}

We thank Mark Tygert for many helpful discussions.

\end{document}